\documentclass[12pt]{amsart}

\usepackage{a4wide}
\usepackage{amssymb}
\usepackage{amsmath}

\usepackage{enumerate}

\newtheorem{theorem}{Theorem}[section]
\newtheorem{lemma}[theorem]{Lemma}
\newtheorem{proposition}[theorem]{Proposition}

\theoremstyle{definition}
\newtheorem{remark}[theorem]{Remark}

\newcommand{\C}{\ensuremath{\mathbb{C}}}
\newcommand{\R}{\ensuremath{\mathbb{R}}}
\renewcommand{\H}{\ensuremath{\mathbb{H}}}
\renewcommand{\O}{\ensuremath{\mathbf{O}}}
\newcommand{\SO}{\ensuremath{\mathbf{SO}}}
\newcommand{\SU}{\ensuremath{\mathbf{SU}}}
\newcommand{\U}{\ensuremath{\mathbf{U}}}
\newcommand{\Sp}{\ensuremath{\mathbf{Sp}}}
\newcommand{\Spin}{\ensuremath{\mathbf{Spin}}}
\newcommand{\g}[1]{\ensuremath{\mathfrak{#1}}}
\newcommand{\cal}[1]{\ensuremath{\mathcal{#1}}}

\newcommand{\Cl}{\ensuremath{\mathrm{Cl}}}
\newcommand{\Aut}{\ensuremath{\mathrm{Aut}}}
\newcommand{\End}{\ensuremath{\mathrm{End}}}
\newcommand{\rank}{\ensuremath{\mathrm{rank}\;}}

\DeclareMathOperator{\Ad}{Ad}

\DeclareMathOperator{\ad}{ad}

\DeclareMathOperator{\Id}{Id}

\begin{document}
\title[Polar foliations on quaternionic projective spaces]{Polar foliations on quaternionic projective spaces}

\author[M.~Dom\'{\i}nguez-V\'{a}zquez]{Miguel Dom\'{\i}nguez-V\'{a}zquez}
\address{Instituto Nacional de Matem\'atica Pura e Aplicada (IMPA), Rio de Janeiro, Brazil.}
\email{mvazquez@impa.br}
\author[C.~Gorodski ]{Claudio Gorodski}
\address{Instituto de Matem\'atica e Estat\'istica, Universidade de S\~ao Paulo, S\~ao Paulo, Brazil.}
\email{gorodski@ime.usp.br}

\thanks{The first named author has been supported by a fellowship by IMPA (Brazil), and projects EM2014/009 and MTM2013-41335-P with FEDER funds (Spain).
The second named author has been partially supported by the
CNPq grant 303038/2013-6 and the FAPESP thematic project 2011/21362-2.}

\subjclass[2010]{Primary 53C12; Secondary 53C35, 57S15}


\begin{abstract}
We classify irreducible polar foliations of codimension $q$ on quaternionic projective spaces $\H P^n$, for all $(n,q)\neq(7,1)$. We prove that all irreducible polar foliations of any codimension (resp.\ of codimension one) on $\H P^n$ are homogeneous if and only if $n+1$ is a prime number (resp.\ $n$ is even or $n=1$). This shows the existence of inhomogeneous examples of codimension one and higher. 
\end{abstract}

\keywords{Polar foliation, singular Riemannian foliation, $s$-representation, symmetric space, FKM-foliation, homogeneous foliation, quaternionic projective space}

\maketitle

\section{Introduction}

A singular Riemannian foliation $\cal{F}$ on a complete Riemannian manifold $M$ is called a \emph{polar foliation} if through each point of $M$ passes a complete, isometrically immersed submanifold, called a section, that intersects all the leaves of $\cal{F}$ and always orthogonally. It turns out that sections are totally geodesic.
A polar foliation is said to be \emph{homogeneous} if it is the family of orbits of an isometric action on $M$; in this case, this action is called a~\emph{polar~action}. 

Polar foliations on nonnegatively curved space forms have been studied under the name of \emph{isoparametric foliations}, which have been almost completely classified after outstanding work by many mathematicians. We refer to the survey~\cite{Th:milan} and to the recent papers~\cite{Chi,Mi:annals} for more information. The question that remains open is to decide
whether a given codimension one polar foliation on the unit sphere $S^{31}$
must be either homogeneous or one of the known inhomogeneous examples
(see also~\cite{Mi:erratum,Siffert} for the case of $S^{13}$, in which there
is some ongoing discussion). 
These inhomogeneous examples, the so-called FKM-foliations, were constructed by Ferus, Karcher and M\"unzner in~\cite{FKM}. We note that the homogeneous polar foliations are well understood. In particular, by Dadok's seminal work~\cite{Dadok}, homogeneous polar foliations on spheres are induced by $s$-representations, that is, by the isotropy representations of symmetric spaces. Moreover, a deep result by Thorbergsson~\cite{Th:annals} states that irreducible polar foliations of codimension at least two on spheres are homogeneous.

The first investigations of not necessarily homogeneous polar foliations on concrete spaces of nonconstant curvature were developed by Lytchak~\cite{Lytchak:GAFA} and the first named author~\cite{Dom:tams}. Lytchak's paper deals with simply connected symmetric spaces of compact type and rank higher than one. He proves that polar foliations of codimension at least three on these spaces must be hyperpolar (i.e.\ polar with flat sections) which, combined with a result by 
Christ (\cite{Christ}; see also~\cite{GH}), implies the homogeneity of all such foliations. Meanwhile, the second paper~\cite{Dom:tams} deals with polar foliations on a family of rank one symmetric spaces, namely, on complex projective spaces. The main result is an almost complete classification of polar foliations on these spaces. Surprisingly enough, from this classification follows the existence of many inhomogeneous irreducible polar foliations, even in codimension higher than one, which reveals a sharp contrast with Thorbergsson's and Lytchak's results.

In this paper we continue the investigation initiated in~\cite{Dom:tams} of polar foliations on rank one symmetric spaces of compact type, and we derive an almost complete classification on quaternionic projective spaces $\H P^n$. This classification is complete in the case of irreducible foliations of codimension higher than one. Here by irreducible we mean that no proper totally geodesic quaternionic submanifold $\H P^k$, $k=0,\dots,n-1$, of $\H P^n$ is a union of leaves of the foliation. 

Any linear quaternionic structure $\g q$ on a Euclidean space $V=\R^{4n+4}$ determines a Hopf fibration $\pi_\g{q}\colon S^{4n+3}\to\H P^n$ from the unit sphere $S^{4n+3}$ of $V$ onto the quaternionic projective space $\H P^n$. We say that $\g q$ \emph{preserves} a singular Riemannian foliation $\cal{F}$ on the sphere $S^{4n+3}$ if each leaf of $\cal{F}$ is foliated by fibers of $\pi_\g{q}$. In such a case there is an induced singular Riemannian foliation $\pi_{\mathfrak q}(\mathcal F)$ on $\H P^n$, and indeed all singular Riemannian foliations of 
$\H P^n$ arise in this way.
Our classification amounts to listing the polar foliations $\cal{F}$ on spheres $S^{4n+3}$ that are preserved by some quaternionic structure, and then determining all congruence classes of quaternionic structures preserving $\cal{F}$. We denote by $N_\cal{S}(\cal{F})$ the number of congruence classes of polar foliations on $\H P^n$ that can be obtained as $\pi_\g{q}(\cal{F})$, for some polar foliation $\cal{F}$ on $S^{4n+3}$ and some quaternionic structure $\g q$ preserving $\cal{F}$. The explicit description is given in Section~\ref{sec:classification}.

In the case of codimension higher than one, each irreducible polar foliation on a sphere must be homogeneous and, more specifically, the orbit foliation of the isotropy representation of an irreducible symmetric space $G/K$, restricted to the unit sphere of $T_{[K]}G/K$. We denote this polar foliation by $\cal{F}_{G/K}$. Thus, we have the following result.

\begin{theorem}\label{th:main1}
For each symmetric space $G/K$ of rank at least two in Table~\ref{table:main1}, there are, up to congruence in $\H P^n$, exactly $N_\cal{S}(\cal{F}_{G/K})$ polar foliations on~$\H P^n$ whose pull-back under the Hopf map gives a foliation congruent to $\cal{F}_{G/K}$.

Conversely, let $\cal{G}$ be an irreducible polar foliation of codimension greater than one on $\H P^n$. Then $\cal{G}$ is the projection of $\cal{F}_{G/K}$ under the Hopf map associated with $\g q$, where $G/K$ is one of the symmetric spaces listed in Table~\ref{table:main1} and $\g{q}$ is a quaternionic structure on $T_{[K]}G/K$ preserving $\cal{F}_{G/K}$.

Moreover, if $G/K$ is an irreducible quaternionic-K\"ahler symmetric space, exactly one of the $N_\cal{S}(\cal{F}_{G/K})$ foliations that pull back to a foliation congruent to $\cal{F}_{G/K}$ is homogeneous. If $G/K$ is not quaternionic-K\"ahler, none of the $N_\cal{S}(\cal{F}_{G/K})$ foliations is homogeneous. 
\end{theorem}
\begin{table}[h]
\begin{tabular}{|c|c|c|}
\hline
$G/K$ & $N_\cal{S}$ & Condition
\\ \hline  
$\SU_{2p+q}/\mathbf{S}(\U_{2p}\times\U_q)$ & $2$ & $q$ even and $q\neq 2p$
\\
& $1$ & $q$ odd or $q=2p$
\\ \hline
$\SO_{4p+q}/\SO_{4p}\times\SO_q$ & $2$ & $q\equiv 0\,(\mathrm{mod}\,4)$ and $q\neq 4p$
\\
& $1$ & $q\not\equiv 0\,(\mathrm{mod}\,4)$ or $q=4p$
\\ \hline
$\Sp_{p+q}/\Sp_p\times\Sp_q$ & $2$ & $p\neq q$
\\
& $1$ & $p=q$
\\ \hline
\end{tabular}
\qquad
\begin{tabular}{|c|c|}
\hline
$G/K$ & $N_\cal{S}$ 
\\ \hline
$\mathbf{E}_6/\SU_6\cdot \SU_2$ & $1$
\\ \hline
$\mathbf{E}_6/\Spin_{10}\cdot \U_1$ & $1$
\\ \hline
$\mathbf{E}_7/\Spin_{12}\cdot \Sp_1$ & $2$
\\ \hline
$\mathbf{E}_8/(\Spin_{16}/\mathbb{Z}_2)$ & $1$
\\ \hline
$\mathbf{E}_8/\mathbf{E}_7\cdot\SU_2$ & $1$
\\ \hline
$\mathbf{F}_4/\Sp_3\cdot\SU_2$ & $1$
\\ \hline
$\mathbf{G}_2/\SU_2\cdot\SU_2$ & $1$
\\
\hline
\end{tabular}
\medskip
\caption{}\label{table:main1}
\end{table}

Investigating codimension one polar foliations on quaternionic projective spaces requires the determination of the quaternionic structures that preserve FKM-foliations, up to congruence of the projected foliations. In this paper we carry out this job for all FKM-foliations satisfying $m_+\leq m_-$ (see below for the explanation of this notation). According to the classification results in spheres, each codimension one polar foliation on $S^{4n+3}$, $n\neq 7$, is either a homogeneous foliation $\cal{F}_{G/K}$ for some rank $2$ symmetric space $G/K$, or an inhomogeneous FKM-foliation satisfying $m_+\leq m_-$. Thus, our study (summarized in Theorem~\ref{th:main2} below) yields the classification of codimension one polar foliations on 
$\H P^n$, for all $n\neq 7$.

In order to understand the condition $m_+\leq m_-$ and Theorem~\ref{th:main2} below, we need to introduce some known facts about FKM-foliations; we refer to~\S\ref{subsec:FKM} for more details. An FKM-foliation $\cal{F}_\cal{P}$ is defined in terms of a symmetric Clifford system $(P_0,\dots,P_m)$ on $\R^{2l}$, but it only depends on the $(m+1)$-dimensional vector space of symmetric matrices $\cal{P}=\mathrm{span}\{P_0,\dots,P_m\}$. The codimension one leaves of $\cal{F}_\cal{P}$ have $g=4$ distinct principal curvatures with multiplicities $(m_+,m_-)=(m, l-m-1)$. The Clifford system $(P_0,\dots,P_m)$ determines a Clifford module on $\R^{2l}$. Let $k$ be the number of its irreducible submodules. If $m\equiv 0\, (\mathrm{mod}\,4)$, there are exactly two equivalence classes of irreducible Clifford modules; in this case, let $k_+$ and $k_-$ be the number of each one of these two classes appearing in the decomposition into irreducible submodules.

\begin{theorem}\label{th:main2}
Let $\cal{F}_\cal{P}$ be an FKM-foliation on $S^{4n+3}$ with $\dim \cal{P}=m+1$ and $2\leq m_+\leq m_-$. Then, up to congruence in $\H P^n$, there are exactly $N_\cal{S}(\cal{F}_\cal{P})\geq 1$ polar foliations on $\H P^n$ that pull back under the Hopf map to a foliation congruent to $\cal{F}_\cal{P}$, where:
\begin{itemize}
\item $N_\cal{S}(\cal{F}_\cal{P})=2$ if $m\equiv 0\, (\mathrm{mod}\,8)$ with both $k_+, k_-\equiv 0\, (\mathrm{mod}\,4)$, or if $m\equiv 1,7\, (\mathrm{mod}\,8)$ with $k\equiv 0\, (\mathrm{mod}\,4)$, or if $m\equiv 2,6\, (\mathrm{mod}\,8)$ with $k$ even, or if $m\equiv 3,4,5\, (\mathrm{mod}\,8)$;
\item $N_\cal{S}(\cal{F}_\cal{P})=1$, otherwise. 
\end{itemize} 

Conversely, let $\cal{G}$ be a polar foliation of codimension one on $\H P^n$. Then $\cal{G}$ is the projection of a polar foliation $\cal{F}$ on $S^{4n+3}$ under the Hopf map associated with $\g q$, where $\g q$ is a quaternionic structure on $\R^{4n+4}$ preserving $\cal{F}$, and:
\begin{itemize}
\item $\cal{F}=\cal{F}_\cal{P}$ is an FKM-foliation satisfying $2\leq m_+\leq m_-$; or
\item $\cal{F}=\cal{F}_{G/K}$ for some symmetric space $G/K$ of rank $2$ in Table~\ref{table:main1}; or
\item $\cal{F}$ is an inhomogeneous polar foliation of codimension one on $S^{31}$ whose hypersurfaces have $g=4$ distinct principal curvatures with multiplicities $(7,8)$.
\end{itemize}
Moreover, if $\cal{F}$ is not homogeneous, then $\cal{G}$ is not homogeneous either.
\end{theorem}

Theorems~\ref{th:main1} and~\ref{th:main2}, together with the explicit determination (obtained in Section~\ref{sec:classification}) of the quaternionic structures preserving each polar foliation on a sphere, give the classification of irreducible polar foliations of codimension~$q$ on quaternionic projective spaces $\H P^n$, with the only exception of $(n,q)=(7,1)$.

Similarly as in the case of complex projective spaces, our classification implies the existence of many inhomogeneous irreducible polar foliations on quaternionic projective spaces, even of codimension higher than one. Indeed, we prove some neat characterizations of those dimensions $n$ for which $\H P^n$ admits inhomogeneous examples. Intriguingly enough, these characterizations are completely analogous to the ones obtained in~\cite{Dom:tams} for the complex case, in spite of the fact that the classification in the present paper is quite different~from~that~in~\cite{Dom:tams}.
\begin{theorem}\label{th:intro_codim1}
Every polar foliation of codimension one on $\H P^n$ is homogeneous if and only if $n$ is even or $n=1$.
\end{theorem}
\begin{theorem}\label{th:intro_primes}
Every irreducible polar foliation on $\H P^n$ is homogeneous if and only if $n+1$ is a prime number.
\end{theorem}

We now give a quick overview of our arguments. First, the possibility of classifying polar foliations on quaternionic projective spaces arises from the fact that a singular Riemannian foliation on $\H P^n$ is polar if and only if its pull-back under the Hopf map is polar in $S^{4n+3}$. Thus, the good knowledge we nowadays have of polar foliations in spheres suggests that it is enough to check if each polar foliation on a sphere $S^{4n+3}$ can be the pull-back of a foliation on $\H P^n$. However, there can be noncongruent polar foliations on $\H P^n$ that pull back to congruent foliations on $S^{4n+3}$. Equivalently, given a fixed foliation $\cal{F}$ on $S^{4n+3}\subset \R^{4n+4}$, there can be different quaternionic structures $\g{q}$ on $\R^{4n+4}$ preserving $\cal{F}$ such that the projected foliations via the corresponding Hopf maps $\pi_\g{q}$ are not congruent in $\H P^n$. 

Determining the set $\cal{S}/\sim$ of all quaternionic structures preserving $\cal{F}$, up to congruence of the projected foliations, turns out to be a completely nontrivial job. 
Our task is then to develop a method to solve this problem. Given any singular Riemannian foliation $\cal{F}$ with closed leaves on $S^{4n+3}$, let $K$ be the maximal connected subgroup of $\SO_{4n+4}$ that leaves invariant each leaf of $\cal{F}$. We show that quaternionic structures preserving $\cal{F}$ are induced by those $\g{su}_2$-subalgebras of $\g{k}$ containing an element that is a complex structure on $\R^{4n+4}$. Combining this with the ideas in~\cite{Dom:tams}, we have a systematic approach to determine the moduli space $\cal{S}/\sim$. Then we apply this method to almost all known polar foliations on spheres, which is enough to obtain the classification stated in Theorems~\ref{th:main1} and~\ref{th:main2}.

Finally, we determine which projected polar foliations on $\H P^n$ are homogeneous. At this point, we revisit Podest\`a and Thorbergsson's classification of polar actions on $\H P^n$~\cite{PT:jdg} by making use of our results. The criterion of homogeneity thus obtained as Theorem~\ref{th:homogeneity} is fundamental to the proofs of Theorems~\ref{th:intro_codim1} and~\ref{th:intro_primes}.

This paper is organized as follows. In Section~\ref{sec:Hopf} we show that, roughly speaking, polar foliations have a good behaviour with respect to the Hopf map. Section~\ref{sec:quaternionic_structures} is devoted to the development of a method to investigate singular Riemannian foliations with closed leaves on quaternionic projective spaces. We apply this method to polar foliations in Section~\ref{sec:classification}, by determining the quaternionic structures that preserve homogeneous polar foliations on spheres (in~\S\ref{subsec:class_homogeneous}), and FKM-foliations with $m_+\leq m_-$ (in~\S\ref{subsec:FKM}). Based on this study, in Section~\ref{sec:proofs} we prove Theorems~\ref{th:main1} and~\ref{th:main2}. Finally, in Section~\ref{sec:homogeneity} we investigate the homogeneity of polar foliations on $\H P^n$ and prove Theorems~\ref{th:intro_codim1} and~\ref{th:intro_primes}.

\medskip

The authors would like to thank Marcos Alexandrino, Andreas Kollross,
and Alexander Lytchak for very useful comments.

\section{Behaviour with respect to the Hopf map}\label{sec:Hopf}
We briefly recall the construction of quaternionic projective space and its Hopf fibration; see~\cite[Chapter~3]{Besse} for details. In what follows, Lie algebras are denoted by gothic letters, and the unit sphere of a Euclidean space $V$ is denoted by $S(V)$. 

Let $V$ be the Euclidean space $\R^{4n+4}$. A $3$-dimensional subspace $\g q$ of $\g{so}(V)=\g{so}_{4n+4}$ is called a \emph{(linear) quaternionic structure} on $V$ if there are elements $J_1,J_2,J_3\in \g{q}$ such that $J_i^2=-\Id$ and $J_iJ_{i+1}=J_{i+2}$ (indices modulo $3$), for $i=1$, $2$, $3$. Note that $\g q$ is then a Lie subalgebra of $\g{so}(V)$ isomorphic to $\g{sp}_1\cong\g{su}_2$. We will denote by $Q$ the connected Lie subgroup of $\SO(V)$ with Lie algebra $\g q$. Clearly, $Q=\{a_0 \Id+a_1 J_1+a_2 J_2+a_3J_3: a_i\in\R, \,\sum_{i=0}^3a_i^2=1\}$.

Any quaternionic structure $\g q$ on $V$ induces a principal fiber bundle with total space the unit sphere $S(V)=S^{4n+3}$, base space the quaternionic projective space $\H P^n$, and structural group $Q\cong \Sp_1\cong \SU_2$. The corresponding fibration $\pi\colon S(V)\to \H P^n$ is called the Hopf map, and its fibers are the totally geodesic $3$-dimensional spheres given by the orbits of the isometric action of $Q$ on $S(V)$. The Fubini-Study metric on $\H P^n$ of constant quaternionic sectional curvature $4$ is the one that makes $\pi$ into a Riemannian submersion.

The following result is the starting point of our arguments.

\begin{proposition}\label{prop:Hopf}
Let $\mathcal{G}$ be a singular Riemannian foliation on $\H P^n$. Then $\cal{G}$ is a polar foliation on $\H P^n$ if and only if its pull-back foliation $\pi^{-1}\cal{G}$ is a polar foliation on $S(V)$. In this case, any section of $\cal{G}$ is a totally geodesic $\R P^k$ in $\H P^n$.
\end{proposition}
\begin{proof}
The necessity has been proved in \cite[Proposition~9.1]{Lytchak:GAFA}. Let us assume that $\pi^{-1}\cal{G}$ is a polar foliation on $S(V)$. Any section $\Sigma$ for $\pi^{-1}\cal{G}$ is horizontal. Since the geodesics in $\Sigma$ are horizontal, they are mapped to geodesics of $\H P^n$, and hence $\pi$ maps $\Sigma$ isometrically onto a section for $\cal{G}$. In particular, $\cal{G}$ is polar. The last assertion follows again from \cite{Lytchak:GAFA}. 
\end{proof}

\begin{remark}
It is known that polar and isoparametric foliations constitute the same subclass of singular Riemannian foliations on spheres, see~\cite[Theorem~2.7 and Claim~2 on p.~1173]{Al:illinois}. Hence, Proposition~\ref{prop:Hopf} and~\cite[Theorem~3.4]{HLO} 
imply that this also happens for quaternionic projective spaces.
(A similar remark applies to complex projective spaces, 
cf.~\cite[Proposition~2.1]{Dom:tams}.) 
\end{remark}
\begin{remark}
In~\cite{Dom:tams} it was proved that isoparametric submanifolds have a good behaviour with respect to the Hopf map $S^{2n+1}\to \C P^n$, and that any isoparametric submanifold of $\C P^n$ is an open part of a complete leaf of an isoparametric foliation that fills the whole $\C P^n$. Whether this is also true in the quaternionic setting remains an open question. According to \cite[Theorem~3.4]{HLO} (and similar calculations as in~\cite[Proposition~2.1]{Dom:tams}), the difficulty consists in showing that sections to an isoparametric submanifold in $\H P^n$ are totally real.
\end{remark}

\section{Quaternionic structures preserving a foliation}\label{sec:quaternionic_structures}
In this section we develop a method to study singular Riemannian foliations with closed leaves on quaternionic projective spaces. 

Let $V=\R^{4n+4}$ and let $\cal{F}$ be a closed foliation on $S(V)$, that is, a singular Riemannian foliation on $S(V)$ such that all its leaves are closed. Let $\rho\colon K\to \SO(V)$ be an effective representation of a Lie group $K$ such that $\rho(K)$ is the maximal connected subgroup of $\SO(V)$ leaving each leaf of $\cal{F}$ invariant. Since $\cal{F}$ is closed, $K$ is compact.

We say that a quaternionic structure $\g q$ on $V$ \emph{preserves} the foliation $\cal{F}$ if for all $p\in S(V)$ the orbit $Q\cdot p$ is contained in the leaf of $\cal{F}$ through $p$. Equivalently, $\g q$ preserves $\cal{F}$ if $\cal{F}$ is the pull-back of a foliation on $\H P^n$ under the Hopf map associated with $\g q$. Similarly (cf.~\cite[\S4.1]{Dom:tams}), a complex structure $J$ on $V$ (i.e.\ $J\in\g{so}(V)\cap \SO(V)$) preserves $\cal{F}$ if the Hopf circle $\{\cos(t) p+\sin(t)Jp:t\in \R\}$ through any $p\in S(V)$ is contained in the leaf of $\cal{F}$ through $p$.

\begin{proposition}\label{prop:rho}
Let $\g q$ be a quaternionic structure on $V$. The following are equivalent:
\begin{enumerate}[{\rm (a)}]
\item $\g q$ preserves $\cal{F}$.
\item There exists a subgroup $S$ of $K$ such that $\rho(S)=Q$.
\item There exists a subalgebra $\g{s}$ of $\g{k}$ such that $\rho_*(\g{s})=\g{q}$.
\end{enumerate}
In this situation, $S\cong \SU_2$ and $\g{s}\cong \g{su}_2$.
\end{proposition}
\begin{proof}
If $\g q$ preserves $\cal{F}$, then $Q \subset \rho(K)$. By the effectiveness of $\rho$, there is a subgroup $S$ of $K$ such that $\rho(S)=Q$. Thus, (a) implies (b). Since the Lie algebra of $Q$ is $\g q$, (b) implies (c). Finally, assume that $\g{s}$ is a subalgebra of $\g{k}$ such that $\rho_*(\g{s})=\g{q}$. Let $S$ be the connected subgroup of $K$ with Lie algebra $\g{s}$. Then, for all $p\in S(V)$ and $X\in \g{s}$ we have $\rho_*(X)p\in T_p L_p$, where $L_p$ is the leaf of $\cal{F}$ through $p$. Since $T_p(\rho(S)\cdot p)=\{\rho_*(X)p:X\in\g{s}\}$, we have that $T_p(\rho(S)\cdot p)\subset T_p L_p$, for all $p\in S(V)$. Thus, $Q\cdot p=\rho(S)\cdot p\subset L_p$ for all $p\in S(V)$, which means that $\g q$ preserves $\cal{F}$. 
\end{proof}

A straightforward but important observation is the following.

\begin{proposition}\label{prop:2-sphere}
If $\g q=\rho_*(\g{s})$ is a quaternionic structure preserving $\cal{F}$, then for each nonzero $X\in\g{s}$ there is $\lambda>0$ such that $\rho_*(\lambda X)$ is a complex structure on $V$ preserving $\cal{F}$. In particular, $\{\rho_*(X):X\in\g{s},\,\rho_*(X)^2=-\Id\}$ is a $2$-sphere of complex structures on $V$ preserving $\cal{F}$.
\end{proposition}

The following characterization of the subspaces of $\g{k}$ that induce a quaternionic structure will be fundamental to our study.

\begin{proposition}\label{prop:sum_of_standard}
Let $\g{s}$ be a subalgebra of $\g{k}$ isomorphic to $\g{su}_2$. The following conditions are equivalent:
\begin{enumerate}[{\rm (a)}]
\item $\rho_*(\g{s})$ is a quaternionic structure on $V$.
\item $\rho_*\rvert_\g{s}\colon \g{s}\to \g{so}(V)$ is the direct sum of standard (i.e.\ nontrivial $4$-dimensional) real representations of $\g{su}_2$.
\item There is an $H\in \g{s}$ such that $\rho_*(H)$ is a complex structure on $V$.
\end{enumerate}
\end{proposition}
\begin{proof}
The implication (a)$\Rightarrow$(b) is direct from the definition of 
quaternionic structure whereas (a)$\Rightarrow$(c)
follows from Proposition~\ref{prop:2-sphere}. Let us assume 
that there is an $H\in \g{s}$ such that $\rho_*(H)$ is a complex structure 
on $V$; in particular $H$ is nonzero. Consider $\R H$ as a maximal Abelian 
subalgebra of $\g{s}$. If $\theta\colon \R H\to \R$ is defined by 
$\theta(aH)=a$, for $a\in\R$, then (c) means that $\pm\theta$ are the 
only weights of the representation $\rho_*\rvert_\g{s}$, and both have the 
same multiplicities; henceforth, for the sake of convenience, we adopt the convention that the weights of representations of compact Lie algebras take real values. By the classification of the orthogonal representations 
of $\g{su}_2$, we deduce that $\rho_*\rvert_\g{s}$ is an orthogonal sum of 
standard representations of $\g{su}_2$. This shows that (c) implies~(b). Finally, 
if (b) holds, then $\rho_*\rvert_\g{s}=\bigoplus_i \rho_i$, where 
$\rho_i\colon \g{s}\to \O(V_i)\cong\O_4$ is a standard representation 
of $\g{su}_2$ and $V=\bigoplus_i V_i$ is a direct sum. Taking 
$\{X_1, X_2, X_3\}$ as a basis of $\g{s}$ such that 
$\{\rho_1(X_1),\rho_1(X_2),\rho_1(X_3)\}$ is a canonical basis 
for a quaternionic structure on $V_1$, we also have that 
$\{\rho_i(X_1),\rho_i(X_2),\rho_i(X_3)\}$ is a canonical basis 
for a quaternionic structure on $V_i$, for any 
$i=1,\dots, \frac{1}{4}\dim V$, since for each $i$ there exists 
an orthogonal transformation $A_i\colon V_1\to V_i$ such 
that $A_i \rho_1 (X) A_i^{-1}=\rho_i(X)$, for all $X\in\g{s}$. 
Thus, 
$\rho_*(\g{s})=\mathrm{span}\{\bigoplus_i\rho_i(X_1),\bigoplus_i\rho_i(X_2),\bigoplus_i\rho_i(X_3)\}$ is a quaternionic structure on $V$, proving~(a).
\end{proof}

Let $\cal{S}$ be the collection of subalgebras $\g{s}$ of $\g{k}$ such that $\rho_*(\g{s})$ is a quaternionic structure on $V$. Clearly, each element of $\cal{S}$ is isomorphic to $\g{su}_2$ and, by Proposition~\ref{prop:rho}, $\{\rho_*(\g{s}):\g{s}\in\cal{S}\}$ is the set of all quaternionic structures on $V$ that preserve $\cal{F}$.

We now analyse the congruence problem, namely, when two quaternionic structures preserving $\cal{F}$ give rise to congruent projected foliations. The basic observation is the following.

\begin{proposition}\label{prop:congruence}
Let $\g{q}_1$, $\g{q}_2$ be quaternionic structures on $V$, $\H P^n_1$, $\H P^n_2$ the corresponding quaternionic projective spaces, and $\pi_1$, $\pi_2$ the associated Hopf maps.

Two foliations $\cal{G}_1\subset \H P^n_1$ and $\cal{G}_2\subset \H P^n_2$ are congruent if and only if there exists an orthogonal transformation $A\in \O(V)$ satisfying $A \g{q}_1 A^{-1}=\g{q}_2$ and mapping leaves of $\pi_1^{-1}\cal{G}_1$ to leaves of $\pi_2^{-1}\cal{G}_2$.
\end{proposition}
\begin{proof}
$\cal{G}_1$ and $\cal{G}_2$ are congruent if and only if their pull-backs are congruent in $S(V)$ by an element $A\in \O(V)$, and $A$ descends to an isometry between $\H P^n_1$ and $\H P^n_2$. But the latter is equivalent to $A \g{q}_1 A^{-1}=\g{q}_2$.
\end{proof}

We introduce an equivalence relation $\sim$ in the set $\cal{S}$ that parametrizes the quaternionic structures on $V$ preserving $\cal{F}$. Two subalgebras $\g{s}_1$, $\g{s}_2\in\cal{S}$ are $\sim$-equivalent if $\pi_1(\cal{F})$ and $\pi_2(\cal{F})$ are congruent foliations on the corresponding quaternionic projective spaces; here $\pi_i\colon S(V)\to \H P^n_i$ is the Hopf map associated with the quaternionic structure $\rho_*(\g{s}_i)$, for $i=1,2$. Thus, the classification (up to congruence in $\H P^n$) of all foliations on $\H P^n$ that pull back under the Hopf map to a foliation congruent to $\cal{F}$ is equivalent to the determination of the moduli space $\cal{S}/\sim$. 

Let $\Aut(\cal{F})$ be the group of automorphisms of $\cal{F}$, that is, the group of all orthogonal transformations of $V$ that map leaves of $\cal{F}$ to leaves of $\cal{F}$. Then, by Proposition~\ref{prop:congruence}, given $\g{s}_1$, $\g{s}_2\in\cal{S}$, we have that $\g{s}_1\sim\g{s}_2$ if and only if there exists $A\in\Aut(\cal{F})$ such that $A\rho_*(\g{s}_1)A^{-1}=\rho_*(\g{s}_2)$. This suggests the introduction, for each $A\in\Aut(\cal{F})$, of an automorphism $\varphi_A\in\Aut(\g{k})$ of the Lie algebra $\g{k}$ by means of the relation $A\rho_*(X)A^{-1}=\rho_*(\varphi_A(X))$ for $X\in\g k$. Thus, we can consider the group
\[
\Aut(\g{k},\cal{F})=\{\Phi_A\in\End(\g{k}\oplus V): \Phi_A\rvert_\g{k}=\varphi_A,\,\Phi_A\rvert_V=A, \, A\in\Aut(\cal{F})\}.
\]
Hence, we have
\begin{proposition}
If $\g{s}\in \cal{S}$, then $\Phi(\g{s})\in\cal{S}$ for all $\Phi\in\Aut(\g{k},\cal{F})$.
Moreover, if $\g{s}_1$, $\g{s}_2\in \cal{S}$, then $\g{s}_1\sim \g{s}_2$ if and only if there exists $\Phi\in\Aut(\g{k},\cal{F})$ such that $\Phi(\g{s}_1)=\g{s}_2$. 
\end{proposition}

Next we fix a maximal Abelian subalgebra $\g{t}$ of $\g{k}$. 
Given any set of simple roots for the pair $(\g{k},\g{t})$, 
let $\bar{C}$ be the closed Weyl chamber in $\g{t}$. We consider the sets $\cal{S}_\g{t}=\{\g{s}\in\cal{S}:\g{s}\cap\g{t}\neq 0\}$  and $\cal{S}_{\bar{C}}=\{\g{s}\in\cal{S}:\g{s}\cap\bar{C}\neq 0\}$. Since $\rho(K)\subset \Aut(\cal{F})$, it follows that $(\Ad\oplus\rho)(K)=\{\Phi_{\rho(k)}:k\in K\}$ is a subgroup of $\Aut(\g{k}, \cal{F})$. Then, since for each $\g{s}\in\cal{S}$ there exists $k\in K$ such that $\Ad(k)\g{s}\in \cal{S}_{\bar{C}}\subset\cal{S}_\g{t}$, and $\Ad(k)\g{s}\sim \g{s}$, we have

\begin{proposition}
$\cal{S}/\!\sim \,\cong \cal{S}_\g{t}/\!\sim \,\cong \cal{S}_{\bar{C}}/\!\sim$.
\end{proposition}

Similarly as in the definition of $\cal{S}$, we consider the set $\cal{J}$ of elements $X\in\g{k}$ such that $\rho_*(X)$ is a complex structure preserving the foliation $\cal{F}$. On $\cal{J}$ we also consider an equivalence relation $\sim$: given $X_1$, $X_2\in \cal{J}$, we say that $X_1\sim X_2$ 
if there is $\Phi\in\Aut(\g{k},\cal{F})$ such that $\Phi(X_1)\in\{\pm X_2\}$. Both $\cal{J}$ and this relation $\sim$ have been studied in \cite{Dom:tams}. Similarly as for $\cal{S}$, we have that $\cal{J}/\!\sim \,\cong \cal{J}\cap\g{t}/\!\sim \,\cong \cal{J}\cap{\bar{C}}/\!\sim$.  

Let $\g{s}\in\cal{S}_{\bar{C}}$. By Proposition~\ref{prop:2-sphere}, there is an $H\in\cal{J}\cap\bar{C}\cap \g{s}$. If there are $H_1$, $H_2\in\cal{J}\cap\bar{C}\cap \g{s}$, $H_1\neq H_2$, then they must be collinear because the rank of $\g{s}$ is one. Since both are in $\cal{J}$, $H_1\in\{\pm H_2\}$, and thus $H_1\sim H_2$. This yields a well-defined map $\cal{S}_{\bar{C}}\to \cal{J}\cap\bar{C}/\!\sim$. This map descends to a map $\cal{S}_{\bar{C}}/\!\sim\to \cal{J}\cap\bar{C}/\!\sim$.  Indeed, let $\g{s}_1$, $\g{s}_2\in \cal{S}_{\bar{C}}$, $\Phi(\g{s}_1)= \g{s}_2$ for some $\Phi\in\Aut(\g{k},\cal{F})$, and $H_i\in\cal{J}\cap\bar{C}\cap\g{s}_i$, $i=1,2$. Then $\Phi(H_1)\in\cal{J}\cap\g{s}_2$ and there exists an element~$k$ in the Lie subgroup of $K$ with Lie algebra $\g{s}_2$ such that $\Ad(k)\Phi(H_1)\in\cal{J}\cap\bar{C}\cap\g{s}_2$. Since $H_2$ and $\Ad(k)\Phi(H_1)$ lie in $\cal{J}\cap\bar{C}\cap\g{s}_2$, we deduce that $H_2\sim\Ad(k)\Phi(H_1)$, and thus~$H_1\sim H_2$. 

This shows the existence of a naturally defined map $\iota\colon\cal{S}_{\bar{C}}/\sim\to \cal{J}\cap\bar{C}/\sim$, which maps a class $[\g{s}]$ with $\g{s}\in\cal{S}_{\bar{C}}$ to the class $\iota([\g{s}])=[H]$, where $H$ is the unique element in $\cal{J}\cap\bar{C}\cap\g{s}$ up to sign. Moreover:
\begin{proposition}\label{prop:injective}
The map $\iota\colon\cal{S}_{\bar{C}}/\!\sim\,\to \cal{J}\cap\bar{C}/\!\sim$ defined above is injective.
\end{proposition}
\begin{proof}
Let $\g{s}_1$, $\g{s}_2\in\cal{S}_{\bar{C}}$ such that $\iota([\g{s}_1])=\iota([\g{s}_2])$. By conjugating one of $\g{s}_1$, $\g{s}_2$ by an element of $\Aut(\g{k},\cal{F})$ if necessary, we can assume that $\dim(\g{s}_1\cap\g{s}_2)\geq 1$ and that there is a nonzero $H\in \bar{C}\cap\g{s}_1\cap\g{s}_2$. By a result of Dynkin about homomorphisms from $\g{su}_2$ to a compact Lie algebra $\g{k}$ (see~\cite[Theorem~7]{Vogan}), we deduce that there is a $k\in K$ such that $\Ad(k)\g{s}_1=\g{s}_2$. Hence $\g{s}_1\sim \g{s}_2$.
\end{proof}

\begin{remark}\label{rem:finite}
Another related result by Dynkin (see~\cite[Theorem~8]{Vogan}) guarantees that the moduli space $\cal{S}/\!\sim$ is finite, since the set of $K$-conjugacy classes of Lie subalgebras of $\g{k}$ isomorphic to $\g{su}_2$ is finite. Geometrically, this means that, given a closed foliation $\cal{F}$ on $S^{4n+3}$, there is at most a finite number of noncongruent foliations on $\H P^n$ that pull back to a foliation congruent to $\cal{F}$ under the Hopf map. Although it was not stated in \cite{Dom:tams}, this result is also true for the complex projective space $\C P^n$. Indeed, as shown in \cite[Proposition~4.2]{Dom:tams}, given $H\in \g{t}$, $\rho_*(H)$ is a complex structure on $V$ preserving $\cal{F}$ if and only if $\lambda(H)\in\{\pm 1\}$ for every weight $\lambda$ of the complexified representation $\rho_*^\C\colon \g{k}\to\g{u}(V^\C)$. Since $\rho$ is effective, the weights of $\rho_*^\C$ span $\g{t}^*$, and thus the number of $H\in\g{t}$ satisfying the condition is finite. The claim follows since any complex structure on $V$ preserving $\cal{F}$ is equivalent to one of the form $\rho_*(H)$ for some $H\in\g{t}$.
\end{remark}

The following result reduces the study of quaternionic structures preserving a 
decomposable foliation to the indecomposable case. Recall that given Euclidean spaces $V_i$, $i=1,\dots,r$, and a foliation $\cal{F}_i$ on the unit sphere of $V_i$ for each $i$, the spherical join $\cal{F}_1\ast\dots\ast\cal{F}_r$ can be described as the restriction of the foliation $\hat{\cal{F}}_1\times\dots\times\hat{\cal{F}}_r$ to the unit sphere of $\bigoplus_{i=1}^r V_i$, where $\hat{\cal{F}}_i$ is the foliation on $V_i$ whose leaves are of the form $r L$, for $r\geq 0$ and $L\in\cal{F}_i$. 
A foliation $\cal F$ on $S(V)$ is called \emph{indecomposable} if it cannot be 
written as a spherical join, and it is called \emph{decomposable} otherwise. 
Every foliation $\cal F$ can be written in an essentially unique way as
a spherical join $\cal F_0\ast\cal{F}_1\ast\dots\ast\cal{F}_r$, where 
$\cal F_0$ consists only of zero-dimensional leaves, and $\cal F_1,\ldots,\cal F_r$
are indecomposable without zero-dimensional leaves, see~\cite{FL} 
and~\cite[Proposition~2.1]{Radeschi}.

\begin{proposition}\label{prop:product}
Let $\cal{F}=\cal F_0\ast\cal{F}_1\ast\dots\ast\cal{F}_r$ be 
as above, where each $\cal{F}_i$ is a closed foliation on the unit sphere of a Euclidean space $V_i$. Let $K$, $\rho$ and $\cal S$ be as above in this section.
Then:
\begin{enumerate}[{\rm (a)}]
\item $K=\prod_{i=0}^r K_i$ for certain subgroups $K_i$ of $K$, where $\rho(K_i)$ is the maximal connected subgroup of $\SO(V)$ that acts trivially on the orthogonal complement of $V_i$ in $V$ and preserves each one of the leaves of $\cal{F}_i$. 
In particular, $K_0$ is the trivial group.
\item $\Aut(\cal F)$ is the subgroup of $\O(V)$ generated by 
$\prod_{i=0}^r\Aut(\cal F_i)$ and all permutations on sets of mutually 
congruent $\cal F_i$. 
\item Let $\cal S_i$ 
be the collection of $\g{su}_2$-subalgebras $\g s_i$ of 
$\g k_i$ such that 
$\rho_*(\g s_i)$ is a quaternionic structure on $V_i$,
for $i=1,\ldots,r$.  
If $\g{s}$ is a subalgebra of $\g{k}$ isomorphic to $\g{su}_2$, then 
$\g{s}\in\cal S$ if and only if $V_0=0$ and $\g{s}_i\in\cal S_i$ for every $i$,  
where $\g{s}_i$ is the image of $\g{s}$ under
the projection of $\g k$ onto $\g{k}_i$. It follows that 
every $\g s\in\cal S$ can be recovered as a diagonal $\g{su}_2$-subalgebra
in $\bigoplus_{i=1}^r\g s_i$ for $\g s_i\in\cal S_i$. 

\end{enumerate}

\end{proposition}
\begin{proof}
Claim~(a) is easy, cf.~\cite[Proposition~4.1(iii)]{Dom:tams}, whereas claim~(b) is a consequence of applying~\cite[Theorem~1.1]{FL} to the space of leaves
$V/\hat{\cal{F}}$ of the product foliation $\hat{\cal F}=\hat{\cal F}_0\times\dots\times\hat{\cal F}_r$.

Next, note that $\cal S\neq\emptyset$ only if $V_0=0$. 
So let $\g{s}\cong\g{su}_2$ be a subalgebra of $\g{k}$. 
Let $\pi_i\colon \g{k}\to\g{k}_i$ be the projection map, for $i=1,\dots, r$. For a fixed $i\in\{1,\dots, r\}$, take any $X\in \g{s}$ and $v\in V_i$, and write $X=\sum_{j=1}^r X_j$, where $X_j=\pi_j(X)\in\g{k}_j$. Then $\rho_*(\pi_i(X))v=\rho_*(X_i)v=\rho_*(X)v$, where the last equality follows from the fact that $\rho(K_j)$ acts trivially on $V_i$ for every $j\neq i$.  
In any case $\pi_i:\g s\to\g s_i$ is an isomorphism
and the representations $\rho_*\rvert_{\g{s}}\colon \g{s}\to \g{so}(V_i)$ and $\rho_*\rvert_{\g{s_i}}\colon \g{s_i}\to \g{so}(V_i)$ are equivalent, for $i=1,\dots, r$.
The equivalence in claim~(c) now follows from Proposition~\ref{prop:sum_of_standard}~(b). 

Finally, note that any two diagonal $\g{su}_2$-subalgebras in $\bigoplus_{i=1}^r\g s_i$, with $\g s_i\in\cal S_i$, are $\sim$-equi\-val\-ent elements of $\cal{S}$ by virtue of the fact that each $\g{s}_i$ admits only inner automorphisms, and $\rho(S_i)\subset\rho(K_i)\subset \Aut(\cal{F}_i)\subset\Aut(\cal{F})$.
\end{proof}

Proposition~\ref{prop:injective} suggests a method to determine the moduli space $\cal{S}/\!\sim\,\cong\cal{S}_{\bar{C}}/\!\sim$ for a fixed closed foliation $\cal{F}$ on $S(V)$ once we know the moduli space $\cal{J}/\!\sim\,\cong\cal{J}\cap\bar{C}/\!\sim$. For each class in $\cal{J}\cap\bar{C}/\!\sim$, take a representative $H\in \cal{J}\cap\bar{C}$ and determine if there exists a Lie subalgebra $\g{s}$ of $\g{k}$ isomorphic to $\g{su}_2$ and containing $H$. If so, then $\g{s}\in\cal{S}_{\bar{C}}$ by Proposition~\ref{prop:sum_of_standard}, and any other element in $\cal{S}$ containing $H$ is $\sim$-equivalent to $\g{s}$; namely, $\iota^{-1}[H]=\{[\g{s}]\}$. If not, then $\iota^{-1}[H]=\emptyset$ and then one should move on to a different class in $\cal{J}\cap\bar{C}/\!\sim$. This way, we determine $\cal{S}/\!\sim$ and, equivalently, the possible foliations on $\H P^n$ that pull back under the Hopf map to a foliation congruent to $\cal{F}$.

In \cite{Dom:tams}, the first author described the set $\cal{J}/\!\sim$ for all irreducible polar foliations $\cal{F}$ on spheres, except for those inhomogeneous codimension one foliations on $S^{31}$ whose hypersurfaces have $4$ principal curvatures with multiplicities $(7,8)$. Therefore, we can carry out the approach described above to determine the set $\cal{S}/\!\sim$ for all irreducible polar foliations $\cal{F}$ on spheres with the exception just mentioned. This is the aim of the next section.

\section{Classification of quaternionic structures preserving polar foliations}\label{sec:classification}
In this section we obtain a case-by-case classification of the quaternionic structures that preserve polar foliations on spheres, up to congruence of the projected foliations of the quaternionic projective space. We will do this for each one of the irreducible homogeneous polar foliations on spheres in \S\ref{subsec:class_homogeneous}, and for each FKM-foliation satisfying $m_+\leq m_-$ in \S\ref{subsec:FKM}.

First of all, note that most of the objects introduced in Section~\ref{sec:quaternionic_structures} (such as $K$, $\rho$, $\cal{S}$ and~$\cal{J}$) depend on the fixed foliation $\cal{F}$, although this dependence has been deleted from the notation to simplify the exposition. It is also important to observe that, as explained in \cite[\S3.1]{Dom:tams}, if we take $(G,K)$ to be an effective symmetric pair of compact type and rank higher than one, with $K$ connected, then $K$ turns out to be the maximal connected subgroup of $\SO(V)$ mapping each leaf of $\cal{F}_{G/K}$ to itself, where $V=T_{[K]}G/K$. Thus, the notation for the isotropy group is coherent with the definition of $K$ in Section~\ref{sec:quaternionic_structures}.

In the sequel we mention some arguments and ideas that will be used throughout our classification. The first observation is that the existence of a quaternionic structure preserving a foliation implies the existence of a complex structure preserving such foliation, as stated in Proposition~\ref{prop:2-sphere}. In~\cite{Dom:tams} it was shown that the only homogeneous polar foliations $\cal{F}_{G/K}$ that admit a complex structure preserving it are the ones induced by inner symmetric spaces $G/K$, that is, those for which $\mathrm{rank}\, G=\mathrm{rank}\, K$. Thus, $\cal{S}=\emptyset$ for polar foliations induced by non-inner symmetric spaces of 
rank higher than one.

Another restriction is of a dimensional nature: a foliation that admits a quaternionic structure must live in a sphere of dimension $4n+3$, for some $n\geq 1$; in particular, if it is a homogeneous foliation $\cal{F}_{G/K}$, then $\dim G/K\equiv 0\,(\mathrm{mod}\, 4)$.

If $G/K$ is an irreducible Hermitian symmetric space, then the center $Z(\g{k})$ of $\g{k}$ is one-dimensional. According to \cite{Dom:tams}, there is exactly one element $H$ in $\cal{J}\cap Z(\g{k})$ up to $\sim$-equivalence, where here $\cal{J}$ parametrizes the complex structures preserving the foliation $\cal{F}_{G/K}$. However, there is no Lie subalgebra of $\g{k}$ isomorphic to $\g{su}_2$ containing $H$. Thus, $\iota^{-1}[H]=\emptyset$. On the other hand, if $G/K$ is an irreducible quaternionic-K\"ahler symmetric space, then $\g{k}$ has a distinguished ideal isomorphic to $\g{su}_2$ satisfying condition (b) in Proposition~\ref{prop:sum_of_standard} (cf.\ comments before Theorem~\ref{th:homogeneity}); thus, this $\g{su}_2$-factor belongs to $\cal{S}$.

We also recall from \cite{Dom:tams} that an element $H$ in a maximal Abelian subalgebra $\g{t}$ of $\g{k}$ belongs to $\cal{J}$ if and only if $\lambda(H)\in\{\pm 1\}$ for every weight $\lambda$ of the complexification $\rho_*^\C$ of $\rho_*$. Note that, if $\rho_*$ is already a complex representation, $H\in\g{t}$ belongs to $\cal{J}$ if and only if $\lambda(H)\in\{\pm 1\}$ for every weight $\lambda$ of $\rho_*$.

Finally, the other fundamental tool we will often use refers to Dynkin's classification of 
conjugacy classes of $\g{su}_2$-subalgebras of a compact Lie algebra. In our situation it will be enough to consider the case of the compact Lie algebra $\g{u}_k$. It turns out that, given an element $X$ in the maximal Abelian subalgebra of diagonal matrices of $\g{u}_k$, if $X$ belongs to a subalgebra of $\g{u}_k$ isomorphic to $\g{su}_2$ then the multiset 
of elements in the diagonal of $X$ is invariant under multiplication by~$-1$.
This follows from the representation theory of $\g{su}_2$. We refer 
to~\cite[Chapter~3]{CoMc} for further information.

In what follows, we will use the following notation. Given a classical Lie algebra $\g{h}\in\{\g{so}_k,\g{u}_k,\g{sp}_k\}$ of rank $r$, we will denote by $\{e_1,\dots,e_r\}$ a canonical basis of the maximal Abelian subalgebra of $\g{h}$. We can and will assume that the maximal Abelian subalgebra consists of diagonal matrices for $\g{u}_k$ and $\g{sp}_k$, and by $2\times 2$-block diagonal matrices for $\g{so}_k$. Thus, if $\g{h}=\g{u}_k$, we take $e_j$, for each $j=1,\dots, r$, to be the matrix with all entries equal to zero with the exception of the entry $(j,j)$ which equals the imaginary unit~$i$; if $\g{h}=\g{so}_k$, we take $e_j$ to be the matrix with all entries equal to zero with the exception of the entries $(2j,2j-1)$ and $(2j-1,2j)$, which are equal to $1$ and $-1$, respectively; and if $\g h=\g{sp}_k$, we view this Lie algebra as a subalgebra 
of $\g{su}_{2k}$ and take $e_j$ to be the matrix with all entries
equal to zero except the entries $(j,j)$ and $(k+j,k+j)$, which are equal to
$i$ and $-i$, respectively. Moreover, 
we denote the dual basis of $\{e_1,\dots,e_r\}$  
by $\{\theta_1,\dots,\theta_r\}$. If we have a sum $\g{h}=\g{h}_1\oplus\g{h}_2$  of two classical algebras, we denote by $\{e_1,\dots,e_{r_1},e_1',\dots,e_{r_2}'\}$ the corresponding basis of the maximal Abelian subalgebra of $\g{h}$, and by $\theta_i$, $\theta_j'$, $i=1,2,\dots,r_1$, $j=1,2,\ldots,r_2$,
the corresponding dual elements.

For each polar foliation $\cal{F}$, we will write $N_{\cal{S}}=N_{\cal{S}}(\cal{F})$ to denote the cardinality of $\cal{S}/\!\sim$, and $N_{\cal{J}}=N_{\cal{J}}(\cal{F})$ for the cardinality of $\cal{J}/\!\sim$ whose values were determined in~\cite{Dom:tams}. 

\subsection{Projecting homogeneous polar foliations}\label{subsec:class_homogeneous}
In this subsection we determine the set $\cal{S}/\!\sim$ for each irreducible homogeneous polar foliation $\cal{F}_{G/K}$, where $G/K$ is an irreducible symmetric space of compact type and rank higher than one. As explained above, we can restrict ourselves to the study of inner symmetric spaces $G/K$, whose isotropy representations can be found in~\cite[Table~8.11.2]{Wolf}. We will run through the cases making use of the set $\cal{J}/\!\sim$ determined in \cite[\S5.3]{Dom:tams}. We label each case by Cartan's notation and by the corresponding orthogonal symmetric pair $(\g{g},\g{k})$.

\newpage
\medskip
{\bf Type A III:} $(\g{su}_{p+q},\g{s}(\g{u}_p\oplus\g{u}_q))$, $p, q\geq 2$.

The isotropy representation is the tensor product $\C^p\otimes_\C (\C^q)^*$ of standard representations. Its weights are $\theta_i-\theta_j'$, with $1\leq i\leq p$, $1\leq j\leq q$. The only elements in $\cal{J}$ up to $\sim$-equivalence are:
\begin{itemize}
\item $-\frac{q}{p+q}\sum_{i=1}^p e_i+ \frac{p}{p+q}\sum_{i=1}^q e_i'\in Z(\g{k})$,
\item $X_{p_1}=\sum_{i=1}^p a_i e_i+\frac{p_2-p_1}{p+q}\sum_{i=1}^q e_i'$, where $a_i=1+\frac{p_2-p_1}{p+q}$ for $i=1,\dots, p_1$, $a_i=-1+\frac{p_2-p_1}{p+q}$ for $i=p_1+1,\dots, p$, and  $p=p_1+p_2$ with $p_1\in\{1,\dots, [\frac{p}{2}]\}$,
\item $Y_{q_1}=\frac{q_2-q_1}{p+q}\sum_{i=1}^p e_i+\sum_{i=1}^q b_i e_i'$, where $b_i=1+\frac{q_2-q_1}{p+q}$ for $i=1,\dots, q_1$, $b_i=-1+\frac{q_2-q_1}{p+q}$ for $i=q_1+1,\dots, q$, and $q=q_1+q_2$ with $q_1\in\{1,\dots, [\frac{q}{2}]\}$.
\end{itemize}
Moreover, if $p=q$, then $X_{p_1}\sim Y_{p_1}$. The element in the center does not belong to any subalgebra of $\g{k}$ isomorphic to $\g{su}_2$. The element $X_{p_1}$ belongs to a subalgebra of $\g{k}$ isomorphic to $\g{su}_2$ if and only if $p_1=p_2$, since it is the only case for which the coefficients of the $e_i$ and $e_i'$ are symmetric with respect to zero. Similarly for $Y_{q_1}$. Thus, $N_\cal{S}=2$ if $p$ and $q$ are different even numbers, $N_\cal{S}=1$ if $p=q$ is even or if only one of $p$ and $q$ is even, and $N_\cal{S}=0$ if both $p$ and $q$ are odd.

\medskip
{\bf Type B I:} $(\g{so}_{2p+2q+1},\g{so}_{2p}\oplus\g{so}_{2q+1}))$, $p+q\geq 3$.
The isotropy representation is the tensor product $\R^{2p}\otimes \R^{2q+1}$ of standard representations. We know that $N_\cal{J}=1$. For dimension reasons, $N_\cal{S}=0$ if $p$ is odd. If $p=2p'$ is even, then we can embed a subalgebra $\g{s}$ isomorphic to $\g{su}_2$ into $\bigoplus_{i=1}^{p'}\g{so}_4\subset\g{so}_{2p}\subset\g{k}$ in a diagonal way. This subalgebra $\g{s}$ contains the element $\sum_{i=1}^{p}\varepsilon_i e_i\in\cal{J}$, where $\varepsilon_i=(-1)^i$. Therefore, $N_\cal{S}=1$ if $p$ is even, and $N_\cal{S}=0$ if $p$ is odd.

\medskip
{\bf Type C I:} $(\g{sp}_p,\g{u}_p)$, $p\geq 2$.

Since this is a Hermitian symmetric space and $N_\cal{J}=1$, we have that $N_\cal{S}=0$.

\medskip
{\bf Type C II:} $(\g{sp}_{p+q},\g{sp}_p\oplus\g{sp}_{q})$, $p,q\geq 2$.

The isotropy representation is the tensor product $\H^p\otimes_\H \H^q$ of standard representations. Any diagonally embedded $\g{s}_p\cong\g{sp}_1\subset\bigoplus_{i=1}^p\g{sp}_1\subset\g{sp}_p$ belongs to $\cal{S}$, since it satisfies the conditions in Proposition~\ref{prop:sum_of_standard}. Similarly, for the other factor $\g{sp}_q$ we obtain an $\g{s}_q\in\cal{S}$. If $p\neq q$, then there is no automorphism of $\g{k}=\g{sp}_p\oplus\g{sp}_q$ mapping $\g{s}_p$ to $\g{s}_q$. If $p\neq q$, then $N_\cal{J}=2$ and so $N_\cal{S}=2$. If $p= q$, we know that $N_\cal{J}=1$, and hence $N_\cal{S}=1$. 

\medskip
{\bf Type D I:} $(\g{so}_{2p+2q},\g{so}_{2p}\oplus\g{so}_{2q})$, $p+q\geq 4$, $q\geq 2$.

The isotropy representation is the tensor product $\R^{2p}\otimes \R^{2q}$ of standard representations. We know that $N_\cal{J}=2$ if $p\neq q$ and $N_\cal{J}=1$ if $p=q$. The only elements in $\cal{J}$ up to $\sim$-equivalence are $X=\sum_{i=1}^p e_i$ and $Y=\sum_{i=1}^q e_i'$, where $X\sim Y$ if $p=q$. If $X$ belongs to a subalgebra $\g{s}$ of $\g{k}$ isomorphic to $\g{su}_2$, then $\g{s}$ is contained in the factor $\g{so}_{2p}$ to which $X$ belongs (since $[X,\g s]\subset\g{so}_{2p}$). The inclusion $\g{s}\subset \g{so}_{2p}$ gives rise to a representation of $\g{su}_2$ on $\R^{2p}$, with weights $\pm \theta$, where $\theta(X)=1$. Thus, this representation is the direct sum of standard representations of $\g{su}_2$. In particular, $p$ must be even. Analogously, we deduce that $Y$ belongs to a subalgebra of $\g{k}$ isomorphic to $\g{su}_2$ if and only if $q$ is even. We conclude that $N_\cal{S}=2$ if $p\neq q$ are even numbers, $N_\cal{S}=1$ if $p=q$ is even or if exactly one of $p$ and $q$ is even, and $N_\cal{S}=0$ if $p$ and $q$ are odd.

\medskip
{\bf Type D III:} $(\g{so}_{2p},\g{u}_p)$, $p\geq 4$.

The isotropy representation is the alternating square $\Lambda^2 \C^p$ of the standard representation of $\g{u}_{p}$. Its weights are $\theta_i+\theta_j$, for $1\leq i<j\leq p$. We know that $N_\cal{J}=2$. Indeed, the only elements in $\cal{J}$ up to $\sim$-equivalence are $\frac{1}{2}(3e_1-\sum_{i=2}^p e_i)$ and $\frac{1}{2}\sum_{i=1}^p e_i$. The first element cannot belong to any subalgebra $\g{s}$ of $\g{su}_p$ isomorphic to $\g{su}_2$, since the coefficients of the $e_i$ are not symmetric with respect to $0$. Finally, $\frac{1}{2}\sum_{i=1}^p e_i\in Z(\g{k})$ does not belong to any subalgebra of $\g{k}$ isomorphic to $\g{su}_2$. Therefore $N_\cal{S}=0$.

\medskip
{\bf Type E II:} $ (\g{e}_6, \g{su}_6\oplus\g{su}_2)$.

The isotropy representation is $\Lambda^3\C^6 \otimes_\H \H$. Its weights are $\theta_i+\theta_j+\theta_k+\theta'_l$, where $1\leq i< j < k \leq 6$, $1\leq l\leq 2$. Then, the only elements in $\cal{J}$ up to $\sim$-equivalence are $\frac{1}{3}(5e_1-\sum_{i=2}^6 e_i)\in\g{su}_6$ and $e_1'-e_2'\in\g{su}_2$. The second one corresponds to the $\g{su}_2$-factor in $\g{k}$, which therefore belongs to $\cal{S}$. The first one, however, cannot belong to any subalgebra $\g{s}$ of $\g{su}_6$ isomorphic to $\g{su}_2$. Hence, $N_\cal{S}=1$.

\medskip
{\bf Type E III:} $(\g{e}_6, \g{so}_{10}\oplus\g{u}_1)$.

The  isotropy representation is the tensor product $\C^{16}\otimes_\C \C$ of the half-spin representation of $\g{so}_{10}$ and the standard representation of $\g{u}_1$. This is a Hermitian symmetric space and $N_\cal{J}=2$. Any canonically immersed subalgebra $\g{s}\cong\g{so}_3$ in $\g{so}_{10}$ is in the conditions of Proposition~\ref{prop:sum_of_standard}, since the restriction $\rho_*\rvert_{\g{s}}$ is a direct sum of spin representations of $\g{so}_3$. This $\g{s}$ yields the only element in $\cal{S}/\!\sim$, so $N_\cal{S}=1$.

\medskip
{\bf Type E V:} $(\g{e}_7, \g{su}_8)$.

The isotropy representation is $[\Lambda^4\C^8]_\R$. Its weights are $\theta_i+\theta_j+\theta_k+\theta_l$, where $1\leq i< j < k < l \leq 8$. Then, the only element in $\cal{J}$ up to $\sim$-equivalence is $\frac{1}{4}(7e_1-\sum_{i=2}^8 e_i)$. But this cannot belong to any subalgebra $\g{s}$ of $\g{k}\cong\g{su}_8$ isomorphic to $\g{su}_2$. Hence, $N_\cal{S}=0$.

\medskip
{\bf Type E VI:} $(\g{e}_7, \g{so}_{12}\oplus\g{su}_2)$.

The isotropy representation is the tensor product $\H^{16}\otimes_\H \H$ of a half-spin representation of $\g{so}_{12}$ and the standard representation of $\g{su}_2$. We have that $N_{\cal{J}}=2$. The $\g{su}_2$-factor in $\g{k}$ gives one element of $\cal{S}$. Any canonically embedded subalgebra $\g{s}\cong\g{so}_3$ in $\g{so}_{12}$ is in the conditions of Proposition~\ref{prop:sum_of_standard}, since the restriction $\rho_*\rvert_{\g{s}}$ is a direct sum of spin representations of $\g{so}_3$. Hence such an $\g{s}\subset\g{so}_{12}$ yields the other element in $\cal{S}/\!\sim$. Thus $N_{\cal{S}}=2$.

\medskip
{\bf Type E VII:} $(\g{e}_7, \g{e}_6\oplus\g{so}_2)$.

Since this is a Hermitian symmetric space and $N_\cal{J}=1$, we have that $N_\cal{S}=0$.
\medskip

{\bf Type E VIII:} $(\g{e}_8, \g{so}_{16})$.

The isotropy representation is the half-spin representation $\R^{128}$ of $\g{so}_{16}$. Hence, any canonically embedded subalgebra $\g{s}\cong\g{so}_3$ in $\g{so}_{16}$ is in the conditions of Proposition~\ref{prop:sum_of_standard}. Since $N_\cal{J}=1$, we deduce that $N_\cal{S}=1$.

\medskip
{\bf Type E IX:} $(\g{e}_8, \g{e}_7\oplus \g{su}_2)$.

This is a quaternionic-K\"ahler symmetric space, and $N_\cal{J}=1$. Hence $\cal{S}$ only contains the $\g{su}_2$-factor in $\g{k}$, and thus $N_\cal{S}=1$.

\medskip

{\bf Type F I:} $(\g{f}_4, \g{sp}_3\oplus\g{su}_2)$.

Same as in the previous case. Thus, $N_\cal{S}=1$.

\medskip
{\bf Type G:} $(\g{g}_2, \g{su}_2\oplus\g{su}_2)$.

Analogous to the previous two cases, $N_\cal{S}=1$. Note that the restriction of the isotropy representation to the distinguished $\g{su}_2$-factor is $2\,\C^2$, 
but to the other one is the symmetric cube $\mathrm{Sym}^3(\C^2)$ of the standard representation $\C^2$, so we get only one quaternionic structure.

\subsection{Projecting FKM-foliations}\label{subsec:FKM}
Next we will determine the set $\cal{S}/\sim$ for each FKM-foliation $\cal{F}_\cal{P}$ satisfying $m_+\leq m_-$. We start by briefly recalling some facts about FKM-foliations, and refer the reader to~\cite{FKM} or to~\cite{Dom:tams} for details. 

A symmetric Clifford system on $\R^{2l}$ is an $(m+1)$-tuple $(P_0,\dots,P_m)$ of symmetric matrices of order $2l$ satisfying the Clifford relations $P_i P_j+P_j P_i=2\delta_{ij} \Id$, for all $i,j=0,\dots, m$, where $\delta_{ij}$ is the Kronecker delta. Each symmetric Clifford system determines an FKM-foliation $\cal{F}_\cal{P}$, which only depends on the $(m+1)$-dimensional vector space of symmetric matrices $\cal{P}=\mathrm{span}\{P_0,\dots,P_m\}$. The regular leaves of $\cal{F}_\cal{P}$ are hypersurfaces with $g=4$ constant principal curvatures with multiplicities $(m_+,m_-)=(m,l-m-1)$. If $\Cl^*_{m+1}$ denotes the Clifford algebra of $\R^{m+1}$ with positive definite quadratic form, it turns out that there is exactly one equivalence class $\g{d}$ of irreducible $\Cl^*_{m+1}$-modules if $m\not\equiv 0 \,(\mathrm{mod}\, 4)$, and two equivalence classes $\g{d}_+$, $\g{d}_-$ if $m\equiv 0\,(\mathrm{mod}\, 4)$. Thus, the Clifford system $(P_0,\dots,P_m)$ determines a representation of $\Cl^*_{m+1}$ on $\R^{2l}$ which is equivalent to $\oplus_{i=1}^k \g{d}$ for some $k$ if $m\not\equiv 0 \,(\mathrm{mod}\, 4)$, or to $(\oplus_{i=1}^{k_+}\g{d}_+)\oplus(\oplus_{i=1}^{k_-}\g{d}_-)$ for some $k_+$, $k_-$ if $m\equiv 0\,(\mathrm{mod}\, 4)$. The integer $k$ and the set $\{k_+,k_-\}$ only depend on $\cal{P}$.

In~\cite[\S3.2]{Dom:tams} the automorphism group and the associated representation $\rho_*$ was calculated for all FKM-foliations satisfying $m_+\leq m_-$. We will make use of this description of $\rho_*$, which varies depending on $m$ $(\mathrm{mod}\,8)$. Thus, in order to determine the set $\cal{S}/\sim$ for such foliations, we will distinguish several cases depending on $m$ $(\mathrm{mod}\,8)$. The moduli space of complex structures $\cal{J}/\!\sim$ has been determined in~\cite[\S6.1]{Dom:tams}. We will denote by $\delta(m)$ half of the real dimension of any irreducible representation of the Clifford algebra $\Cl^*_{m+1}$; in particular, $\delta(1)=1$, $\delta(2)=2$ and $\delta(m)\equiv 0 \,(\mathrm{mod}\, 4)$ for any $m\geq 3$. Moreover, $l=k\delta(m)$ if $m\not\equiv 0 \,(\mathrm{mod}\, 4)$, and $l=(k_++k_-)\delta(m)$ if $m\equiv 0 \,(\mathrm{mod}\, 4)$. We also let $p$ be such that $m+1=2p$ if $m+1$ is even, and $m+1=2p+1$ if $m+1$ is odd.

\medskip

{\bf Type $m\equiv 0\,(\mathrm{mod}\,8)$.}

$\rho_*$ is the direct sum $(\R^{2\delta(m)}\otimes_\R \R^{k_+})\oplus(\R^{2\delta(m)}\otimes_\R \R^{k_-})$ of the tensor products of the spin representation of $\g{so}_{m+1}$ and the standard representations of $\g{so}_{k_\pm}$. The weights of $\rho_*^\C$ are $\frac{1}{2}(\pm \theta_1\pm\dots\pm\theta_p)\pm\theta'^+_j$, $j=1,\dots,q_+$, and $\frac{1}{2}(\pm \theta_1\pm\dots\pm\theta_p)\pm\theta'^-_j$, $j=1,\dots,q_-$, where $q_\pm$ are given by $k_\pm=2q_\pm$ if $k_\pm$ is even, or by $k_\pm=2q_\pm+1$ if $k_\pm$ is odd, and additionally the weight $\frac{1}{2}(\pm \theta_1\pm\dots\pm\theta_p)$ if $k_+$ or $k_-$ is odd. 
The only elements in $\cal{J}$ up to $\sim$-equivalence are $2e_1\in\g{so}_{m+1}$ and, if both $k_+$ and $k_-$ are even, $Y=\sum_{i=1}^{q^+} e_i'^+ + \sum_{i=1}^{q^-} e_i'^-\in\g{so}_{k_+}\oplus\g{so}_{k_-}$. The first element always belongs to a subalgebra of $\g{so}_{m+1}$ isomorphic to $\g{so}_3\cong\g{su}_2$. If $Y$ belongs to a subalgebra of $\g{so}_{k_+}\oplus\g{so}_{k_-}$ isomorphic to $\g{su}_2$, this  determines a representation $\R^{k_+}\oplus\R^{k_-}$ of $\g{su}_2$ with weights $\pm\theta$, where $\theta(Y)=1$. Then, such representation is a sum of standard representations. Since both $\R^{k_+}$ and $\R^{k_-}$ are invariant subspaces, we deduce that $k_+$ and $k_-$ are multiples of $4$. In this case, the existence of a subalgebra of $\g{so}_{k_+}\oplus\g{so}_{k_-}$ isomorphic to $\g{su}_2$ and containing $Y$ is clear. 
In conclusion, $N_\cal{S}=2$ if $k_+, k_-\equiv 0\,(\mathrm{mod}\,4)$, and $N_\cal{S}=1$ if $k_+$ or $k_-\not\equiv 0\,(\mathrm{mod}\,4)$.

\medskip
{\bf Type $m\equiv 1,7\,(\mathrm{mod}\,8)$.}

$\rho_*$ is the tensor product $\R^{2\delta(m)}\otimes_\R \R^k$ of the spin representation of $\g{so}_{m+1}$ and the standard representation of $\g{so}_k$. The weights of $\rho_*^\C$ are $\frac{1}{2}(\pm \theta_1\pm\dots\pm\theta_p)\pm\theta'_j$, $j=1,\dots,q$, where $q$ is given by $k=2q$ if $k$ is even, and by $k=2q+1$ if $k$ is odd, together with $\frac{1}{2}(\pm \theta_1\pm\dots\pm\theta_p)$ if $k$ is odd. The only elements in $\cal{J}$ up to $\sim$-equivalence are $2e_1\in\g{so}_{m+1}$ and, if $k$ is even, $Y=\sum_{i=1}^q e_i'\in\g{so}_k$. The first element belongs to a subalgebra of $\g{so}_{m+1}$ isomorphic to $\g{so}_3\cong\g{su}_2$, whenever $m>1$. If $Y$ belongs to a subalgebra of $\g{so}_k$ isomorphic to $\g{su}_2$, this would determine a representation of $\g{su}_2$ with weights $\pm\theta$, where $\theta(Y)=1$, and thus such representation would be a sum of standard representations. In particular, $k$ would be multiple of $4$. In this case, the existence of a subalgebra of $\g{so}_k$ isomorphic to $\g{su}_2$ and containing $Y$ is clear. In conclusion, $N_\cal{S}=2$ if $m>1$ and $k\equiv 0\,(\mathrm{mod}\,4)$, $N_\cal{S}=1$ if $m>1$ and $k\not\equiv 0\,(\mathrm{mod}\,4)$, or if $m=1$ and $k\equiv 0\,(\mathrm{mod}\,4)$, and $N_\cal{S}=0$ if $m=1$ and $k\not\equiv 0\,(\mathrm{mod}\,4)$.

\medskip
{\bf Type $m\equiv 2,6\,(\mathrm{mod}\,8)$.}

$\rho_*$ is the tensor product $\C^{\delta(m)}\otimes_\C \C^k$ of the spin representation of $\g{so}_{m+1}$ and the standard representation of $\g{u}_k$. The weights of $\rho_*^\C$ are $\frac{1}{2}(\pm \theta_1\pm\dots\pm\theta_p)\pm\theta'_j$, $j=1,\dots,k$. The only elements in $\cal{J}$ up to $\sim$-equivalence are $2e_1\in\g{so}_{m+1}$ and $Y_r=\sum_{i=1}^k \varepsilon_i e_i'\in\g{u}_k$, where $\varepsilon_i=1$ for $i=1,\dots, r$, $\varepsilon_i=-1$ for $i=r+1,\dots, k$, and $r=0, \dots, \left[\frac{k}{2}\right]$. The first element always belongs to a subalgebra of $\g{so}_{m+1}$ isomorphic to $\g{so}_3\cong\g{su}_2$. Among the vectors $Y_r\in\cal{J}$, the only one that belongs to a subalgebra of $\g{u}_k$ isomorphic to $\g{su}_2$ is $Y_{k/2}$ for $k$ even. Therefore, $N_\cal{S}=2$ if $k$ is even, and $N_\cal{S}=1$ if $k$ is odd.

\medskip
{\bf Type $m\equiv 3,5\,(\mathrm{mod}\,8)$.}

$\rho_*$ is the tensor product $\H^{\delta(m)/2}\otimes_\H \H^k$ of the spin representation of $\g{so}_{m+1}$ and the standard representation of $\g{sp}_k$. The weights of $\rho_*^\C$ are $\frac{1}{2}(\pm \theta_1\pm\dots\pm\theta_p)\pm\theta'_j$, $j=1,\dots,k$. The only elements in $\cal{J}$ up to $\sim$-equivalence are $2e_1\in\g{so}_{m+1}$ and $Y=\sum_{i=1}^k e_i'\in\g{sp}_k$. The first element always belongs to a subalgebra of $\g{so}_{m+1}$ isomorphic to $\g{so}_3\cong\g{su}_2$, whereas $Y$ belongs to a diagonally embedded subalgebra $\g{sp}_1$ of $\g{sp}_k$. Therefore, $N_\cal{S}=2$.

\medskip
{\bf Type $m\equiv 4\,(\mathrm{mod}\,8)$.}

$\rho_*$ is the direct sum $(\H^{\delta(m)/2}\otimes_\H \H^{k_+})\oplus(\H^{\delta(m)/2}\otimes_\H \H^{k_-})$ of the tensor products of the spin representation of $\g{so}_{m+1}$ and the standard representations of $\g{sp}_{k_\pm}$. The weights of $\rho_*^\C$ are $\frac{1}{2}(\pm \theta_1\pm\dots\pm\theta_p)\pm\theta'^+_j$, $j=1,\dots,k_+$, and $\frac{1}{2}(\pm \theta_1\pm\dots\pm\theta_p)\pm\theta'^-_j$, $j=1,\dots,k_-$. The only elements in $\cal{J}$ up to $\sim$-equivalence are $2e_1\in\g{so}_{m+1}$ and $Y=\sum_{i=1}^{k_+} e_i'^+ + \sum_{i=1}^{k_-}e_i'^-\in\g{sp}_{k_+}\oplus\g{sp}_{k_-}$. The first element always belongs to a subalgebra of $\g{so}_{m+1}$ isomorphic to $\g{so}_3\cong\g{su}_2$, whereas $Y$ belongs to a diagonal subalgebra $\g{sp}_1$ of $\g{sp}_{k_+}\oplus\g{sp}_{k_-}$. Hence, $N_\cal{S}=2$.

\section{Proofs of Theorems~\ref{th:main1} and~\ref{th:main2}}\label{sec:proofs}
The study carried out in the previous section allows us to conclude the proofs of Theorems~\ref{th:main1} and~\ref{th:main2} stated in the introduction. We start with an easy observation about the irreducibility of the foliations.

\begin{remark}\label{rem:irreducible}
We say that a polar foliation $\cal{G}$ on $\H P^n$ is irreducible if there is no proper totally geodesic quaternionic projective subspace $\H P^k$, $k\in\{0,\dots,n-1\}$, in $\H P^n$ which is foliated by leaves of $\cal{G}$. It is known that a polar foliation $\cal{F}$ on a sphere is indecomposable as a spherical join if and only if it is irreducible, in the sense that there is no proper totally geodesic subspace foliated by leaves of $\cal{F}$. If $\cal{F}=\pi^{-1}\cal{G}$ is the pull-back of a foliation on $\H P^n$, then any proper totally geodesic subspace of $S^{4n+3}$ foliated by leaves of $\cal{F}$ must be foliated as well by the fibers of the Hopf fibration $\pi$ (cf.~Proposition~\ref{prop:product}). Thus, a polar foliation $\cal{G}$ on $\H P^n$ is irreducible if and only if its pull-back $\pi^{-1}\cal{G}$ is irreducible in $S^{4n+3}$. Moreover, Proposition~\ref{prop:product} reduces the study of reducible polar foliations on $\H P^n$ to the irreducible case.
\end{remark}

\begin{proof}[Proof of Theorem~\ref{th:main1}]
The first claim follows from the calculation of $N_\cal{S}$ carried out in~\S\ref{subsec:class_homogeneous}. 

Now let $\cal{G}$ be an irreducible polar foliation with codimension at least two on $\H P^n$. By Proposition~\ref{prop:Hopf} and Remark~\ref{rem:irreducible}, the pull-back $\pi^{-1}\cal{G}$ of $\cal{G}$ under any Hopf map $\pi\colon S^{4n+3}\to \H P^n$ is an irreducible polar foliation of codimension at least two on the sphere $S^{4n+3}$. By Thorbergsson's theorem~\cite{Th:annals}, $\pi^{-1}\cal{G}$ is congruent to the orbit foliation $\cal{F}_{G/K}$ of the isotropy representation $K\times T_{[K]}G/K\to T_{[K]}G/K$ of an irreducible symmetric space $G/K$. Equivalently, there exists a quaternionic structure $\g{q}$ on $T_{[K]}G/K$ preserving $\cal{F}_{G/K}$ such that $\cal{G}$ is congruent to the projection of $\cal{F}_{G/K}$ by the Hopf map induced by $\g{q}$. By the study in \S\ref{subsec:class_homogeneous} the only symmetric spaces $G/K$ that admit such a quaternionic structure are the ones in Table~\ref{table:main1}. This proves the second claim. 

Finally, the assertions about the homogeneity of the projected foliations will follow from Theorem~\ref{th:homogeneity} in Section~\ref{sec:homogeneity}.
\end{proof}

\begin{proof}[Proof of Theorem~\ref{th:main2}]
The first claim summarizes the study of the quaternionic structures preserving FKM-foliations with $m_+\leq m_-$ developed in \S\ref{subsec:FKM}.

Now let $\cal{G}$ be a polar foliation of codimension one on $\H P^n$. Let $\cal{F}$ be the pull-back of $\cal{G}$ under any Hopf map. By Proposition~\ref{prop:Hopf}, $\cal{F}$ is a polar foliation of codimension one on $S^{4n+3}$. By M\"unzner's structure results for codimension one polar foliations on spheres~\cite{Mu1,Mu2}, all regular leaves of $\cal{F}$ are isoparametric hypersurfaces with the same number $g\in\{1,2,3,4,6\}$ of constant principal curvatures, and their multiplicities satisfy $m_i = m_{i+2}$ (indices modulo~$g$). If $g\in\{1,2,3\}$, the corresponding foliation is homogeneous according to Cartan's results~\cite{Cartan}. If $g=6$, Abresch~\cite{Abresch} showed that $m_1=m_2\in\{1,2\}$. On the one hand, Dorfmeister and Neher~\cite{DN} proved the homogeneity of the foliations with $(g,m_1,m_2)=(6,1,1)$; on the other hand $\cal{F}$ cannot satisfy $(g,m_1,m_2)=(6,2,2)$, since the corresponding foliation would live in the sphere $S^{13}$, and $13\not\equiv 3$ (mod $4$). Finally, if $\cal{F}$ satisfies $g=4$, the classification results by Cecil, Chi, Jensen~\cite{CCJ}, Immervoll~\cite{Im} and Chi~\cite{Chi} imply that $\cal{F}$ is homogeneous or of FKM-type, except maybe if the hypersurfaces of $\cal{F}$ satisfy $(g,m_1,m_2)=(4,7,8)$. By Dadok's work~\cite{Dadok}, if $\cal{F}$ is homogeneous, it is of the type $\cal{F}_{G/K}$ for some symmetric space $G/K$ of rank $2$.

If $\cal{F}$ is an FKM-foliation with $m_+=1$, then it is homogeneous~\cite[\S6.1]{FKM}. Now, up to congruence, there are exactly eight FKM-foliations with $m_+>m_-$ (see~\cite[\S4.3 and \S5.5]{FKM} or~\cite[\S3.2]{Dom:tams}). However, each one of these foliations is either homogeneous or congruent to another FKM-foliation with $m_+\leq m_-$, except for two examples with $(m_+,m_-)=(8,7)$.

Altogether, we have that $\cal{F}$ is homogeneous of the type $\cal{F}_{G/K}$, or an FKM-foliation with $2\leq m_+\leq m_-$ or an inhomogeneous foliation on $S^{31}$ whose hypersurfaces satisfy $(g,m_1,m_2)=(4,7,8)$. This implies the second claim in the theorem.

The last assertion of Theorem~\ref{th:main2} is easy (cf.~Section~\ref{sec:homogeneity}).
\end{proof}

\section{Inhomogeneous polar foliations}\label{sec:homogeneity}
In this section we derive the existence of inhomogeneous polar foliations on quaternionic projective spaces.

A polar foliation is homogeneous if it is the orbit foliation of a polar action. Podest\`a and Thorbergsson~\cite{PT:jdg} classified polar actions on quaternionic projective spaces up to orbit equivalence. In the first part of this section, we revisit their result in our context.

An elementary observation is that each homogeneous polar foliation on $\H P^n$ must be the projection under the Hopf map of a homogeneous polar foliation $\cal{F}_{G/K}$ on $S^{4n+3}$. Our first aim is to determine when the projection to $\H P^n$ of a homogeneous polar foliation $\cal{F}_{G/K}$ on $S^{4n+3}$ is homogeneous. We start with a basic lemma and  use the 
notation of Section~\ref{sec:quaternionic_structures}.

\begin{lemma}\label{lemma:normalizer}
Let $\cal{F}$ be a closed foliation on $S(V)$, $\g{s}\in\cal{S}$ and $\pi\colon S(V)\to \H P^n$ the Hopf map associated with the quaternionic structure $\rho_*(\g{s})$. Let $S$ be the connected subgroup of $K$ with Lie algebra $\g{s}$ and $N_K^0(S)$ the identity connected component of the normalizer of $S$ in~$K$.

Then $\pi(\cal{F})$ is a homogeneous foliation on $\H P^n$ if and only if the orbit foliation of the action of $\rho(N_K^0(S))$ on $S(V)$ coincides with $\cal{F}$. 
\end{lemma}
\begin{proof}
Each isometry $[A]$ of $\H P^n$ is induced by an isometry $A\in \O(V)$ of $S(V)$ that maps orbits of $\rho(S)$ to orbits of $\rho(S)$. Equivalently, such an $A$ satisfies $A \rho(S) A^{-1}x=\rho(S)x$ for all $x\in S(V)$. This means that the group $A \rho(S) A^{-1}$ acts with the same orbits as $\rho(S)$. 

We claim that $\rho(S)$ is the maximal connected subgroup of $\SO(V)$ acting with the fibers of $\pi$ as orbits. Suppose the contrary. Then the larger group
would act with nontrivial principal isotropy group and yield a reduction
of the $\rho(S)$-action (cf.~\cite[\S2.3]{GL} or~\cite[\S2.6]{GL2}), 
contradicting~\cite[Proposition~1.1]{GL}. We deduce from the claim
that $A \rho(S) A^{-1}=\rho(S)$, i.e.\ $A\in N_{\O(V)}(\rho(S))$.

Thus, the maximal connected subgroup of isometries of $\H P^n$ that leave invariant each leaf of $\pi(\cal{F})$ is the identity connected component of $\rho(K)\cap N_{\O(V)}(\rho(S))$, which is $\rho(N_K^0(S))$. This implies the lemma.
\end{proof}

It is appropriate to recall here that a Riemannian manifold of dimension $4k$ is called quaternionic-K\"ahler if its holonomy group is a subgroup of $\Sp_k\cdot\Sp_1$. A symmetric space $G/K$ without Euclidean factor is quaternionic-K\"ahler if and only if it is irreducible and its isotropy group $K$ contains a normal subgroup isomorphic to $\SU_2=\Sp_1$ whose isotropy representation is equivalent to the representation of the normal subgroup $\Sp_1$ of $\Sp_k\cdot\Sp_1$ when considered with its standard action on $\H^k$. We refer to the Lie algebra of this normal subgroup of $K$ as the distinguished $\g{su}_2$-factor of $\g{k}$. For more details, see~\cite[\S14.E]{Besse:Einstein}.

We also recall that when we deal with a homogeneous polar foliation $\cal{F}_{G/K}$ induced by an effective symmetric pair $(G,K)$ of compact type, with $K$ connected, we can regard the representation $\rho$ as the adjoint representation $\Ad\colon K\to \O(\g{p})$, and then $\rho_*=\ad$ is the adjoint representation on the Lie algebra level. Here $\g{p}$ is the orthogonal complement of $\g{k}$ in $\g{g}$ with respect to the Killing form of $\g{g}$. Thus, $\g{g}=\g{k}\oplus\g{p}$ is the decomposition into eigenspaces of the involution induced
by the geodesic symmetry of $G/K$. 

We can now provide a different approach to the proof of the classification of homogeneous polar foliations on $\H P^n$ due to Podest\`a and Thorbergsson~\cite{PT:jdg}. Our arguments are based on the classification of quaternionic structures obtained in~\S\ref{subsec:class_homogeneous} and on the explicit description of all connected groups acting polarly on spheres (due to Eschenburg and Heintze~\cite{EH:mz} for the irreducible case, and to Bergmann~\cite{Be:mm} and Fang, Grove and Thorbergsson~\cite{FGT} for the reducible case).

\begin{theorem}\label{th:homogeneity}
Let $(G,K)$ be a compact, effective symmetric pair, $(G,K)=\prod_{i=1}^r(G_i,K_i)$ its decomposition in irreducible factors, and $\g{g}_i=\g{k}_i\oplus\g{p}_i$ the decomposition 
into eigenspaces of the involution associated to $G_i/K_i$. Let $\g q=\ad(\g{s})\rvert_\g{p}$ be a quaternionic structure on $\g{p}=\bigoplus_{i=1}^r\g{p}_i$ that preserves the foliation $\cal{F}_{G/K}$, where $\g{s}\cong \g{su}_2$ is a subalgebra of $\g{k}$. Then, the following conditions are equivalent:
\begin{enumerate}[{\rm (i)}]
\item The projection of $\cal{F}_{G/K}$ to the quaternionic projective space $\H P^n$ determined by $\g q$ is a homogeneous foliation.

\item All but maybe one of the irreducible factors of $G/K$ have rank one, and the possible exception $G_j/K_j$, for some $j\in\{1,\dots,r\}$, is a quaternionic-K\"ahler symmetric space such that the projection of $\g{s}$ on $\g{k}_j$ yields the distinguished $\g{su}_2$-factor of $\g{k}_j$.
\end{enumerate}
\end{theorem}
\begin{proof}
Note that we can change the symmetric pairs 
associated to rank one factors to those of the type $(\SO_{4p_i+1},\SO_{4p_i})$, where $4p_i=\dim \g{p}_i$, if necessary. Indeed, the foliation $\cal{F}_{G/K}$ remains the same, and $\g{s}$ remains in $\cal{S}$.

First suppose (ii). Let $\g{s}_i$ be the projection of $\g{s}$ on $\g{k}_i$, for each $i=1,\dots, r$. As a closed subgroup of a compact Lie group, we have that $N_K^0(S)=Z_K^0(S)\cdot S$. It is clear also that $Z_K^0(S)=\prod_{i=1}^r Z_{K_i}^0(S_i)$. If $i\neq j$, then $Z^0_{K_i}(S_i)$ is isomorphic to $\Sp_{p_i}$. On the other hand, it is clear that $Z_{\g{k}_j}(\g{s}_j)=\g{k}_j'$, where $G_j/K_j$ is a quaternionic-K\"ahler irreducible symmetric space and $\g{k}_j=\g{k}_j'\oplus\g{su}_2$. Hence $Z^0_{K_j}(S_j)=K_j'$. All in all, $N_K^0(S)=(\prod_{i=1}^r Z_{K_i}^0(S_i))\cdot S\cong((\prod_{i\neq j}\Sp_{p_i})\times K_j')\cdot S$. Then, $N_K^0(S)$ acts on $S(\g{p})$ by the adjoint representation with the same orbits as the adjoint action of $K$ and, thus, Lemma~\ref{lemma:normalizer} guarantees the homogeneity of the projected foliation.

Now assume (i).
By Proposition~\ref{prop:product}, the projection of $\g{s}$ onto each $\g{k}_i$, $i=1,\dots,r$, is a subalgebra $\g{s}_i$ of $\g{k}_i$ isomorphic to $\g{su}_2$, and $\g q_i=\ad(\g{s}_i)\rvert_{\g{p}_i}$ is a quaternionic structure on $\g{p}_i$ preserving $\cal{F}_{G_i/K_i}$. Each subalgebra $\g{s}_i\subset\g{k}_i$ determines a Hopf fibration $\pi_i\colon S(\g{p}_i)\to \H P^{\dim\g{p}_i/4-1}$. The projection of each irreducible foliation $\cal{F}_{G_i/K_i}$, $i=1,\dots,r$, via the Hopf map $\pi_i$, is homogeneous as follows from the assumption. Then by Lemma~\ref{lemma:normalizer}, $N^0_{K_i}(S_i)$ is a connected subgroup of $K_i$ acting on $S(\g{p}_i)$ by the adjoint representation with the same orbits as $K_i$. By \cite{EH:mz}, whenever $G_i/K_i$ has rank at least two, any connected subgroup $K_i'$ of $K_i$ whose adjoint action has the same orbits as $K_i$ must be the whole $K_i$, except for a few possibilities. After excluding cases (i), (iv) with $pq$ odd, and (v) in \cite{EH:mz} (because the corresponding foliations cannot be preserved by any quaternionic structure according to our classification in \S\ref{subsec:class_homogeneous}), we are left with the following possibilities:
\begin{itemize}
\item $(G_i,K_i)=(\SO_{10},\SO_2\times\SO_8)$, $K_i'=\SO_2\times \Spin_7$, and $S_i$ is, up to conjugation, a diagonally embedded $\Sp_1\subset \SO_8\subset K_i$. Then $N^0_{K_i}(S_i)=\SO_2\times(\Sp_2\cdot\Sp_1)\neq \SO_2\times\Spin_7$.
\item $(G_i,K_i)=(\SO_{11},\SO_3\times\SO_8)$, $K_i'=\SO_3\times \Spin_7$, and $S_i$ is, up to conjugation, a diagonally embedded $\Sp_1\subset \SO_8\subset K_i$. Then $N^0_{K_i}(S_i)=\SO_3\times(\Sp_2\cdot\Sp_1)\neq \SO_3\times\Spin_7$.
\item $(G_i,K_i)=(\SU_{p+q},\mathbf{S}(\U_p\times\U_q))$ with $p\neq q$, $p$ or $q$ even, $K_i'=\SU_p\times\SU_q$, and $S_i$ is, up to conjugation, a diagonally embedded $\SU_2\subset \SU_p\subset K_i$ if $p=2p'$ is even. Then $N^0_{K_i}(S_i)=\mathbf{S}(((\Sp_{p'}\cdot\Sp_1)\cap\U_p)\times \U_q)\neq\SU_p\times\SU_q$. Similarly if $q$ is even.
\item $(G_i,K_i)=(\mathbf{E}_{6},\Spin_{10}\cdot\U_1)$, $K_i'=\Spin_{10}$, and $S_i$ is, up to conjugation, a canonically embedded $\Spin_3\subset \Spin_{10}\subset K_i$. Then $N^0_{K_i}(S_i)=(\Spin_3\times\Spin_7)\cdot \U_1\neq\Spin_{10}$.
\end{itemize}
Altogether, Lemma~\ref{lemma:normalizer}, the result in~\cite{EH:mz} and the homogeneity assumption imply that  $N^0_{K_i}(S_i)=K_i$. Hence, if $G_i/K_i$ has rank higher than one, then $S_i$ is a normal subgroup of $K_i$; since $\ad(\g{s}_i)\rvert_{\g{p}_i}$ is a quaternionic 
structure on~$\g p_i$, it follows that $G_i/K_i$ is quaternionic-K\"ahler.

Now assume that there are at least two irreducible factors of $G/K$ with rank higher than one; let them be $G_1/K_1$ and $G_2/K_2$. Decompose $K_i$ as a direct product $S_i\cdot L_i$ of normal subgroups, $i=1,2$. The group $(L_1\times L_2)\cdot S$ acts almost effectively on $\g{p}_1\oplus\g{p}_2$. Then, by 
\cite[Lemma~9.4]{FGT}, its adjoint action on $\g{p}_1\oplus\g{p}_2$ is standard, which in this case means that $(L_1\times L_2)\cdot S$ has the same orbits as $L_1\times L_2$. In particular, the adjoint actions of $K_i$ and $L_i$ on $\g{p}_i$ are orbit equivalent, for $i=1,2$. Again, by \cite{EH:mz}, both $G_1/K_1$ and $G_2/K_2$ cannot have rank higher than one, which contradicts our assumption. Therefore, at most one irreducible factor $G_j/K_j$ of $G/K$ has rank greater than one.
\end{proof}

We can now analyse the existence of inhomogeneous polar foliations 
depending on the dimension of the ambient quaternionic projective space. 
We start by restricting to the codimension one case, thus obtaining the 
characterization stated in Theorem~\ref{th:intro_codim1} in the introduction. This result turns out to be completely analogous to the one derived in~\cite[Theorem~7.4(i)]{Dom:tams} for codimension one polar foliations on complex projective spaces (cf.~\cite[Theorem~1.1]{GTY}). Although very similar arguments to the ones used in \cite{Dom:tams} would work here as well, we prefer to include a different proof based on the classification we obtained above in this paper. 

For the proofs of Theorems~\ref{th:intro_codim1} and~\ref{th:intro_primes} we will need the explicit values of the rank and dimension of the different symmetric spaces~\cite[p.~518]{Helgason}. 

\begin{proof}[Proof of Theorem~\ref{th:intro_codim1}]
We start by proving the necessity. Let $n\geq 3$ be odd. Consider the symmetric space $G/K=\SU_{n+3}/\mathbf{S}(\U_{n+1}\times \U_2)$. By our classification in~\S\ref{subsec:class_homogeneous} and by the characterization of the homogeneity of projected foliations in Theorem~\ref{th:homogeneity}, we know that there is an irreducible inhomogeneous polar foliation $\cal{G}$ on $\H P^n$ whose pull-back under the Hopf map is congruent to $\cal{F}_{G/K}$. Since $\rank G/K=2$, the codimension of $\cal{G}$ in $\H P^n$~is~one.

Next we show the other implication. First, it is well-known that $\H P^1\cong S^4$ only admits homogeneous polar foliations, see~\cite{Cartan,Th:annals} or~\cite{Radeschi}. Then let us assume that $n$ is even and let $\cal{G}$ be a codimension one polar foliation on $\H P^n$. We will show that $\cal{G}$ is homogeneous.

By Theorem~\ref{th:main2}, we know that $\cal{G}$ is the projection via some Hopf map of a homogeneous foliation $\cal{F}_{G/K}$ for some rank two symmetric space, or of an inhomogeneous FKM-foliation $\cal{F}_\cal{P}$ with $m_+\leq m_-$, or we are in the open case $n=7$. Actually, we cannot be in this last case, since we are assuming that $n$ is even.

Assume first $\cal{G}$ is the projection of a foliation $\cal{F}_{G/K}$ for some rank two symmetric space such that $N_\cal{S}\geq 1$. According to our classification and to the dimensions and ranks of the symmetric spaces, since  $\dim G/K = 4(n+1)$ with $n+1$ odd, the only possibilities for $G/K$ are $\SU_{n+3}/\mathbf{S}(\U_{n+1}\times \U_2)$ or a reducible symmetric space $S^{4r}\times S^{4(n-r+1)}$ for some $r=1,\dots, n$. But since in all these cases we have $N_\cal{S}=1$, Theorem~\ref{th:homogeneity} guarantees that $\cal{G}=\pi(\cal{F}_{G/K})$ is homogeneous.

Next assume $\cal{G}$ is the projection of an inhomogeneous FKM-foliation $\cal{F}_\cal{P}$ with $m_+\leq m_-$. Since we take $\cal{F}_\cal{P}$ not homogeneous, we can assume that $m=m_+\geq 3$ (see~\cite[\S4.4]{FKM}). It turns out that, in this case, $m_++m_-\equiv 3\,(\mathrm{mod}\,4)$, as follows from the facts stated at the beginning of \S\ref{subsec:FKM}. However, since the regular leaves of $\cal{F}_\cal{P}$ are hypersurfaces in $S^{4n+3}$ with four principal curvatures, two of them with multiplicity $m_+$ and the other two with multiplicity $m_-$, we must also have $4(n+1)=2(m_++m_- +1)\equiv0\,(\mathrm{mod}\,8)$. This is a contradiction to the assumption that $n$ is even. Hence $\cal{G}$ cannot be the projection of an inhomogeneous FKM-foliation. This concludes the proof.
\end{proof}

Finally we prove Theorem~\ref{th:intro_primes}, which provides a characterization of those dimensions $n$ for which $\H P^n$ admits some irreducible inhomogeneous polar foliation of some codimension. Intriguingly, and similarly as for $\C P^n$, prime numbers appear again in this characterization. However, the proof of the analogous result for $\C P^n$  in~\cite{Dom:tams} does not carry over directly to the quaternionic setting.

\begin{proof}[Proof of Theorem~\ref{th:intro_primes}]
Suppose $n+1$ is not prime. Write $n+1=pq$ for integers $p$, $q\geq 2$. Consider the irreducible symmetric space $G/K=\Sp_{p+q}/\Sp_p\times\Sp_q$, which is not quaternionic-K\"ahler. By \S\ref{subsec:class_homogeneous} we know that $N_\cal{S}\geq 1$ for this symmetric space. Then, there is, up to congruence in $\H P^n$, at least one polar foliation $\cal{G}$ on $\H P^n$ whose pull-back under the Hopf map is congruent to $\cal{F}_{G/K}$. This $\cal{G}$ is irreducible, and cannot be homogeneous due to Theorem~\ref{th:homogeneity}. 

Now let $n+1$ be prime and let $\cal{G}$ be an irreducible polar foliation on $\H P^n$. If $\cal{G}$ has codimension one, then we know by Theorem~\ref{th:intro_codim1} that $n$ is odd and $n\geq 3$, or $\cal{G}$ is homogeneous. The first case is impossible because we are assuming that $n+1$ is prime. Hence, $\cal{G}$ is homogeneous.

Let us then assume that $\cal{G}$ has codimension at least two. Since it is irreducible, it must be obtained by projecting some homogeneous polar foliation $\cal{F}_{G/K}$ for some irreducible symmetric space $G/K$ such that $\rank G/K\geq 3$ and $N_\cal{S}\geq 1$. Moreover, $\dim G/K=4(n+1)$ is four times a prime number. By appealing again to \S\ref{subsec:class_homogeneous} and to the dimensions and ranks of the symmetric spaces, we are left with the following possibilities for $G/K$: $\SO_{n+5}/\SO_{n+1}\times \SO_4$ and $\mathbf{F}_4/\Sp_3\cdot\SU_2$. But these symmetric spaces are quaternionic-K\"ahler and satisfy $N_\cal{S}=1$. Hence, by Theorem~\ref{th:homogeneity}, their projections must be homogeneous. Thus, $\cal{G}$ is homogeneous.
\end{proof}


\end{document}